\documentclass[11pt]{amsart}

\usepackage{amsmath}
\usepackage{amssymb}
\usepackage{fancybox}
\usepackage{eucal}
\usepackage{amsthm}
\usepackage{amscd}
\usepackage{wasysym}
\usepackage{url}
\usepackage{MnSymbol} %smybols for infinitary formulas
\usepackage{amssymb,}
\usepackage[]{amsmath, amsthm, amsfonts,graphicx, amscd,}
\usepackage[all,cmtip]{xy}
\usepackage{setspace, xspace}

\usepackage{enumitem}% http://ctan.org/pkg/enumitem

\usepackage{hyperref}

\newcommand{\ACVF}{\mathit{ACVF}}
\newcommand{\pCF}{\mathit{pCF}}

\DeclareMathSymbol{\mlq}{\mathord}{operators}{``}
\DeclareMathSymbol{\mrq}{\mathord}{operators}{`'}

\DeclareMathOperator{\cl}{cl}

\DeclareMathOperator{\acl}{acl}

\DeclareMathOperator{\divi}{div}

\newcommand{\DIV}[0]{\divi}

\newtheorem {theorem}{Theorem}[section]

\newtheorem{proposition}[theorem]{Proposition}

\newtheorem {lemma}[theorem]{Lemma}

\newtheorem{claim}{Claim}

\theoremstyle{remark}

\newtheorem{np*}{Non-Proof}

\theoremstyle{definition}
\newtheorem{definition}[theorem]{Definition}

\numberwithin{subcase}{case}

\newcommand{\mc}{\mathcal }

\begin{document}

\setlist[enumerate]{noitemsep, topsep=0pt}

\title{Computable Valued Fields}

\author[M. Harrison-Trainor]{Matthew Harrison-Trainor}
\address{Group in Logic and the Methodology of Science\\
University of California, Berkeley\\
 USA}
\email{matthew.h-t@berkeley.edu}
\urladdr{\href{http://www.math.berkeley.edu/~mattht/index.html}{www.math.berkeley.edu/$\sim$mattht}}

\thanks{The author was partially supported by the Berkeley Fellowship and NSERC grant PGSD3-454386-2014.}

\begin{abstract}
We investigate the computability-theoretic properties of valued fields, and in particular algebraically closed valued fields and $p$-adically closed valued fields. We give an effectiveness condition, related to Hensel's lemma, on a valued field which is necessary and sufficient to extend the valuation to any algebraic extension. We show that there is a computable formally $p$-adic field which does not embed into any computable $p$-adic closure, but we give an effectiveness condition on the divisibility relation in the value group which is sufficient to find such an embedding. By checking that algebraically closed valued fields and $p$-adically closed valued fields of infinite transcendence degree have the Mal{\textquotesingle}cev property, we show that they have computable dimension $\omega$.
\end{abstract}

\maketitle

% van den Dries' result (constructivized by Scowcroft, using a lemma of Denef) that the theory of p-adically closed elds admits denable Skolem functions.

% http://etd.lsu.edu/docs/available/etd-04112014-114109/unrestricted/thesis.pdf

% http://www.ams.org/mathscinet/search/publdoc.html?arg3=&co4=AND&co5=AND&co6=AND&co7=AND&dr=all&pg4=AUCN&pg5=TI&pg6=ALLF&pg7=ALLF&pg8=ET&review_format=html&s4=&s5=&s6=computable&s7=valued%20field&s8=All&vfpref=html&yearRangeFirst=&yearRangeSecond=&yrop=eq&r=5&mx-pid=357377

\section{Introduction}

Recently there has been interest in studying, from the perspective of computability theory, various types of fields which arise in model theory. Marker and Miller \cite{MarkerMiller} studied the degree spectra of differentially closed fields, while Miller, Ovchinnikov, and Trushin \cite{MillerOvchinnikovTrushin14} have looked at generalizations of splitting algorithms for differential fields. Real closed fields have been studied by Calvert \cite{Calvert04}, Ocasio \cite{Ocasio}, Knight and Lange \cite{KnightLange13}, and Igusa, Knight, and Schweber \cite{IgusaKnightSchweber}. Generalizations to difference fields of Rabin's theorem on embeddings into algebraic closures have been studied by Melnikov, Miller, and the author \cite{HTMelnikovMiller15}. This article is a study of valued fields from the perspective of computable algebra. Variations of Rabin's theorem for valued fields were previously studied by Smith \cite{Smith}; some of our results extend those of that paper.

\begin{definition}
A valued field is a field $K$ together with a \textit{valuation} $v$ on $K$, that is, a map $K \to \Gamma \cup \{\infty\}$ from $K$ to an ordered abelian group $\Gamma$, such that
\begin{enumerate}
	\item $v(x) = \infty$ if and only if $x = 0$,
	\item $v(xy) = v(x) + v(y)$, and
	\item $v(x + y) \geq \min(v(x),v(y))$ (with equality if $v(x) \neq v(y)$).
\end{enumerate}
$\Gamma$ is called the \textit{value group}. We will always assume that the valuation is surjective.
\end{definition}
\noindent Standard examples of valued fields are the $p$-adic valuations on $\mathbb{Q}$ and their completions, the $p$-adic fields $\mathbb{Q}_p$.

In computable algebra, we consider computable presentations of algebraic structures. A computable valued field is a field whose underlying domain is a computable set $K \subseteq \omega$, equipped with computable functions $+_K$ and $\times_K$ giving the addition and multiplication operations, and with a computable valuation, i.e.\, a computable function $v \colon K \to \Gamma$ where $\Gamma$ is a computable group (a computable subset of $\omega$ with a computable group operation). There are a number of equivalent ways of presenting a valued field (see Section \ref{sec:lang}), but this method is most faithful to the classical definition of a valued field. Two computable valued fields may be classically isomorphic but not computably isomorphic.

One objective of computable algebra is to see which classical theorems hold in the effective setting, considering only computable objects. For example, it is a classical result that every valued field embeds into an algebraically closed valued field. The same is true in the effective setting: every computable valued field effectively embeds into a computable algebraically closed valued field. Similarly, every valued field has a Henselization, and every computable valued field effectively embeds into a computable presentation of its Henselization.

On the other hand, a slight variation of this does not hold. If we fix an embedding of a valued field $(K,v)$ into its algebraic closure $(\overline{K},w)$ with an extension of the valuation, the Henselization of $K$ in $\overline{K}$ is unique. In the effective setting, we assume that these fields $(K,v)$ and $(\overline{K},w)$ are computable and that the embedding is effective. In this case, we cannot compute the Henselization of $K$ inside of $\overline{K}$, even if we assume that $K$ has a splitting algorithm (an algorithm for finding the minimal polynomial over $K$ of an element of $\overline{K}$, or equivalently, for deciding which elements of $\overline{K}$ are actually in $K$).  Thus there is no effective criteria to decide, for a given $a \in \overline{K}$, and using only the minimal polynomial of $a$ over $K$ and the valuations of various elements, whether or not $a$ is in the Henselization of $K$.

\subsection{Extending Valuations}

In \cite{HTMelnikovMiller15} the author, together with Melnikov and Miller, considered the problem of extending an automorphism of a field $F$ to an automorphism of an algebraic extension $K$ of $F$ (with a fixed computable embedding of $F$ in $K$). In this article, we consider the related problem of extending a valuation of $F$ to a valuation of $K$. Smith \cite{Smith} proved several results along these lines, most importantly that every valued field embeds into an algebraically closed field with an extension of the valuation, but that one cannot do this with a fixed embedding into a fixed algebraically closed field.  Our main result is as follows:
\begin{theorem}
Let $(K,v)$ be a computable algebraic valued field. Then the following are equivalent:
\begin{enumerate}
	\item for every computable embedding $\iota \colon K \to L$ of $K$ into a field $L$ algebraic over $K$, there is a computable extension of $v$ to a computable valuation $w$ on $L$,
	\item the \textit{Hensel irreducibility set}
\begin{align*}
H_{K} := \{ f = x^n + a_{n-1} x^{n-1} + a_{n-2} x^{n-2} + \cdots + a_0 \in \mc{O}_K[x] : \\ 
 f \text{ is irreducible over } K\text{, } v(a_{n-1}) = 0 \text{, and } v(a_{n-2}),\ldots,v(a_0) > 0 \}
\end{align*}
	of $(K,v)$ is computable.
\end{enumerate}
\end{theorem}

\subsection{\texorpdfstring{$p$-adically Closed Fields}{p-adically Closed Fields}}

Among the most important examples of valued fields are the $p$-adics $\mathbb{Q}_p$. The theory of $p$-adically closed fields is the theory of $\mathbb{Q}_p$. Just as the theory of real closed fields is the model companion of the formally real fields, the theory of $p$-adically closed fields is the model companion of a class of fields called the formally $p$-adic fields. Classically, every formally $p$-adic embeds into a $p$-adic closure. The effective analogue is false:
\begin{theorem}
There is a computable formally $p$-adic field which does not embed into a computable $p$-adic closure.
\end{theorem}
\noindent The issue is that we can construct a formally $p$-adic field in which the divisibility relation on the value group is not computable. If we have an algorithm to compute the divisibility relation on the value group of a formally $p$-adic field, then we can effectively embed that field into a computable $p$-adic closure.
\begin{theorem}
Let $(K,v)$ be a computable formally $p$-adic valued field with value group $\Gamma$. Suppose that we can compute, for each $\gamma \in \Gamma$ and $k \in \mathbb{N}$, whether $\gamma$ is divisible by $k$. Then there is a computable embedding of $K$ into a computable $p$-adic closure $(L,w)$.
\end{theorem}

\subsection{Copies with Computable and Non-Computable Transcendence Bases}

Many algebraic structures admit a notion of independence, such as algebraic independence in field, linear independence in vectors spaces, $\mathbb{Z}$-linear independence in abelian groups, and differential independence in differential fields. In the 1960's, Mal{\textquotesingle}cev noticed that there are two non-computably-isomorphic computable presentations of the infinite-dimensional $\mathbb{Q}$-vector space, one with a computable basis, and the other with no computable basis, and that the two were $\Delta^0_2$-isomorphic. Many other structures have been found to have the same property, such as algebraically closed fields, torsion-free abelian groups \cite{Nurtazin74b,Dobritsa83,Goncharov82}, Archimedean ordered abelian groups \cite{GoncharovLemppSolomon03}, differentially closed fields, real closed fields, and difference closed fields \cite{HTMelnikovMontalban15}. In \cite{HTMelnikovMontalban15}, the author together with Melnikov and Montalb\'an formally characterized this phenomenon (which they named the Mal{\textquotesingle}cev property) using the notion of a r.i.c.e.\ pregeometry, and presented a metatheorem unifying all of these examples. Here we will apply the metatheorem to algebraically closed valued fields and $p$-adically closed valued fields.
\begin{theorem}
Every computable algebraically closed valued field or $p$-adically closed valued field $K$ of infinite transcendence degree has a computable copy $G \cong_{\Delta^0_2} K$ with a computable transcendence base and a computable copy $B \cong_{\Delta^0_2} K$ with no computable transcendence base.
\end{theorem}
\noindent Note that by a theorem of Goncharov \cite{Goncharov82}, every such structure has computable dimension $\omega$.

\section{Preliminaries}

\subsection{Splitting algorithms}

Recall that the \textit{splitting set} $S_{F}$ of ${F}$ is the set of all polynomials $p \in {F}[X]$ which are reducible over ${F}$. The splitting set of a field is not necessarily computable (see \cite[Lemma 7]{Miller08}), but it is always c.e. If the splitting set of ${F}$ is computable, then we say that ${F}$ has a \textit{splitting algorithm}. Finite fields and algebraically closed fields trivially have splitting algorithms. Kronecker \cite{Kronecker82} showed that $\mathbb{Q}$ has a splitting algorithm, and also that many other field extensions also have splitting algorithms:

\begin{theorem}[Kronecker \cite{Kronecker82}; see also \cite{vanderWaerden70}]\label{Kronecker}
The field $\mathbb{Q}$ has a splitting algorithm. If a computable field ${F}$ has a splitting algorithm, and $a$ is transcendental over ${F}$ (or separable and algebraic over ${F}$), then ${F}(a)$ has a splitting algorithm. Moreover, in the case that $a$ is algebraic over ${F}$, the splitting algorithm for ${F}(a)$ can be found uniformly in the splitting algorithm for ${F}$ and the minimal polynomial of $a$ over ${F}$. If $a$ is transcendental over ${F}$, then the splitting algorithm can be found uniformly in the splitting algorithm for ${F}$. 
\end{theorem}

\noindent Given a field ${F}$ with a splitting algorithm and an element $a$ which is either transcendental over ${F}$, or separable and algebraic over ${F}$, we know that ${F}(a)$ has a splitting algorithm. However, the algorithm depends on whether $a$ is transcendental or algebraic. To find a splitting algorithm uniformly, we must know which is the case.

Rabin \cite{Rabin60} showed that every computable field ${F}$ has a computable algebraic closure $\overline{{F}}$, and moreover there is a computable embedding $\imath \colon {F} \to \overline{{F}}$. We call such an embedding a \textit{Rabin embedding}. Moreover, he characterized the image of ${F}$ under this embedding:

\begin{theorem}[Rabin \cite{Rabin60}]
Let ${F}$ be a computable field. Then there is a computable algebraically closed field $\overline{{F}}$ and a computable field embedding $\imath \colon {F} \to \overline{{F}}$ such that $\overline{{F}}$ is algebraic over $\imath({F})$. Moreover, for any such $\overline{{F}}$ and $\imath$, the image $\imath({F})$ of ${F}$ in $\overline{{F}}$ is Turing equivalent to the splitting set of ${F}$. 
\end{theorem}

\subsection{Valued fields}

The \textit{valuation ring} $\mc{O}_{K,v}$ of $K$ is the subring consisting of all elements $a$ with $v(a) \geq 0$. $\mc{O}_{K,v}$ is a local ring with maximal ideal $\mathfrak{m}_{K,v} = \{ x : v(x) > 0 \}$. The \textit{residue field} $k_{K,v}$ is the quotient $\mc{O}_{K,v} / \mathfrak{m}_{K,v}$. When the valuation $v$ is clear from the context, we write $\mc{O}_K$, $\mathfrak{m}_K$, and $k_K$. Given $a \in \mc{O}_K$, we denote by $\bar{a}$ its image in the residue field. For a comprehensive reference on valued fields, see \cite{EnglerPrestel05}.

\begin{definition}\label{def:Hensel}
A valued field $(K,v)$ is \textit{Henselian} if it satisfies one of the following equivalent properties (see \cite[Theorem 4.1.3]{EnglerPrestel05}):
\begin{enumerate}
	\item $v$ has a unique extension to every algebraic extension $L$ of $K$,
	\item given $f \in \mc{O}_K[x]$ and $a \in \mc{O}_K$ such that $v(f(a)) > 2v(f'(a))$, there is a unique $b \in \mc{O}$ such that $f(b) = 0$ and $v(a - b) > v(f'(a))$,
	\item given $f \in \mc{O}_K[x]$ and $a \in \mc{O}_K$ such that $\bar{f}(\bar{a}) = 0$ and $\bar{f}'(\bar{a}) \neq 0$, there is a $b \in \mc{O}_K$ with $f(b) = 0$ and $\bar{a} = \bar{b}$,
	\item every polynomial $x^n + a_{n-1} x^{n-1} + a_{n-2}x^{n-2} + \cdots + a_0 \in \mc{O}_K[x]$ with $v(a_{n-1}) = 0$ and $v(a_{n-2}),\ldots,v(a_0) > 0$ has a solution in $K$.
\end{enumerate}
\end{definition}
Every valued field has a Henselization, that is, a minimal Henselian field into which it embeds. The Henselization of a field is algebraic over that field, and every Henselization of a given field is isomorphic. Moreover, after fixing an embedding of the field into its algebraic closure, the Henselization is unique. We denote by $K^h$ the Henselization of a field $K$.

If $(L,w)$ is a valued field extension of $(K,v)$, then we may view the value group $\Gamma_K$ as a subgroup of $\Gamma_L$ and the residue field $k_K$ as a subfield of $k_L$. We call $e(w / v) = [\Gamma_L : \Gamma_K]$ the \textit{ramification index} of the extension and $f(w / v) = [k_L : k_K]$ the \textit{residue degree} of the extension. An extension is called \textit{immediate} if the ramification index and the residue degree are both $1$. If we consider a field $L$ which is an extension (as a field) of the valued field $(K,v)$, we can ask about extensions of $v$ to $L$. There may in general be many possible extensions, but the number is limited by the degree $[L:K]$ of the extension according to the following theorem.

\begin{theorem}[Theorems 3.3.4 and 3.3.5 of \cite{EnglerPrestel05}]\label{thm:fund-ineq}
Let $L/K$ be a finite extension of fields and $v$ a valuation on $K$. Let $w_1,\ldots,w_n$ be the distinct extensions of $v$ to $L$. Then
\[ \sum_{i=1}^n e(w_i/v) f(w_i/v) \leq [ L : K ].\]
If the extension $L / K$ is separable and the value group of $K$ is $\mathbb{Z}$, then we have equality.
\end{theorem}

\noindent This inequality is known as the \textit{fundamental inequality}. In the case that we have equality, i.e., when the extension is separable and the value group is $\mathbb{Z}$, we call this the \textit{fundamental equality}. All of the extensions $w_1,\ldots,w_n$ in the theorem are conjugate by an automorphism of $L$ over $K$.

The following theorem will allow us to represent extensions of a valuation across a finite extension $L/K$ of fields by elements of $L$. It is a restatement of Theorem 3.2.7 (3) of \cite{EnglerPrestel05} for finite extensions of fields, using Lemma 3.2.8 to see that the hypotheses of Theorem 3.2.7 can be simplified in this case.

\begin{theorem}\label{thm:distinguish}
Let $L / K$ be a finite extension of fields and $v$ a valuation on $K$. Let $w_1,\ldots,w_n$ be distinct valuations on $L$ extending $v$. Then given $a_1,\ldots,a_n \in L$ such that $w_i(a_i) \geq 0$ for all $i$, there is $a \in L$ such that $w_i(a) \geq 0$ for all $i$ and $w_i(a-a_i) > 0$ for all $i$.
\end{theorem}

Let $(K,v)$ be a valued field. If $(K,v)$ has no proper separable immediate extensions, then $K$ is Henselian. We call such a $K$ \textit{algebraically maximal}. The converse is only true if $K$ is \textit{finitely ramified}: if the residue field has characteristic zero, or if it has characteristic $p$ and there are only finitely many elements of the value group between $0$ and $1 = v(p)$.

\begin{theorem}[Theorem 4.1.10 of \cite{EnglerPrestel05}]\label{thm:alg-max}
Suppose that $(K,v)$ is finitely ramified. Then $(K,v)$ is Henselian if and only if it is algebraically maximal.
\end{theorem}

\subsection{Computable valued fields}
\label{sec:lang}

There are many natural languages in which to talk about valued fields \cite{Chatzidakis11}. Three of them are:
\begin{enumerate}
	\item Macintyre's language $\mc{L}_{\text{div}}$ which adds a binary relation $a \mid b$ to the ring language, with $a \mid b$ interpreted as $v(a) \leq v(b)$.
	\item Robinson's two-sorted language $\mc{L}_{\text{Rob}}$ which has a sort for the value group (as an ordered group) and contains the valuation function $v \colon K \to \Gamma \cup \infty$.
	\item The three-sorted language $\mc{L}_{\Gamma,k}$ which extends $\mc{L}_{\text{Rob}}$ by adding the residue field and residue map.
\end{enumerate}
A computable valued field is a computable field (i.e., the domain is a computable set, and the operations of addition and multiplication are computable) together with a computable valuation. By this we mean, in $\mc{L}_{\text{div}}$, that the relation $a \mid b$ is computable; in $\mc{L}_{\text{Rob}}$, that there is a computable group $\Gamma$ and that the valuation map $v$ is computable; and in $\mc{L}_{\Gamma,k}$, that in addition the residue field $k$ and the residue map are computable. It follows from the proof of the following proposition that all three ways of presenting a valued field are \textit{effectively bi-interpretable} (see \cite{HTMelnikovMillerMontalban}), and hence it does not matter which we choose.

\begin{proposition}\label{prop:comp-value-gp}
Let $(K,v)$ be a computable valued field in the language $\mc{L}_{\mathrm{div}}$. There is a computable presentation $\Gamma$ of the value group of $K$ and a computable presentation $k$ of the residue field of $K$ so that the valuation map $v \colon K \to \Gamma$ and the reduction map $\mc{O}_K \to k$ are computable.
\end{proposition}
\begin{proof}
The value group $\Gamma$ is the quotient of $K^{\times}$ by the computable equivalence relation
\[ a \sim b \Longleftrightarrow (a \mid b) \wedge (b \mid a).\]
The group operation is given by $[a] + [b] = [ab]$. The ordering on the value group is that induced by $a \mid b$. The valuation map $v \colon K \to \Gamma$ is just the quotient map.

We can compute, inside $K$, the valuation ring $\mc{O}_K$. The residue field is the quotient of $\mc{O}_K$ by its maximal ideal $\mathfrak{m} = \{ a \in \mc{O}_K : v(a) > 0 \}$. So we can present the residue field as a quotient of the valuation ring by the computable equivalence relation
\[ a \sim b \Longleftrightarrow v(a - b) > 0.\qedhere\]
\end{proof}

\subsection{Algebraically closed valued fields}

The theory $\ACVF$ of \textit{algebraically closed valued fields} is axiomatized by saying that $(K,v)$ is a valued field which is algebraically closed as a field (and recalling that we assumed that the valuation map is surjective). For a reference on algebraically closed valued fields, see \cite{Chatzidakis11}. The theory is complete (after naming the characteristic and the characteristic of the residue field), decidable, and admits quantifier elimination. $\ACVF$ is the model completion of the theory of valued fields.

\subsection{\texorpdfstring{$p$-adically closed valued fields}{p-adically closed valued fields}}

A valued field $(K,v)$ extending $\mathbb{Q}$ is \textit{formally $p$-adic} if:
\begin{enumerate}
	\item $v$ extends the $p$-adic valuation on $\mathbb{Q}$,
	\item the residue field is $\mathbb{F}_p$, and
	\item $v(p)$ is the least positive element of the value group.
\end{enumerate}
$K$ is \textit{$p$-adically closed} if in addition:
\begin{enumerate}[resume]
	\item $K$ is Henselian and
	\item the value group is elementarily equivalent to $\mathbb{Z}$, i.e., a model of Presburger arithmetic.\footnote{The models of Presburger arithmetic are the discrete ordered abelian semigroups with a zero and a least element $1$, such that for all $x$ and $n$ there is $y$ such that $x = ny + r$ for some $r = 0,\ldots,n-1$.}
\end{enumerate}
This axiomatizes the complete theory $\pCF$ of $p$-adically closed fields, which is the theory of the $p$-adics $\mathbb{Q}_p$. See \cite{PrestelRoquette84} for a reference on formally $p$-adic fields.

In a formally $p$-adic field, we can identify $\mathbb{Z}$ with the convex subgroup of the value group $\Gamma$ generated by $v(p)$. The \textit{coarse valuation} $\bar{v}$ is the composition of $v$ with the quotient map $\Gamma \to \Gamma / \mathbb{Z}$. Then $\Gamma$ is elementarily equivalent to $\mathbb{Z}$ if and only if $\Gamma / \mathbb{Z}$ is divisible. We call $\Gamma / \mathbb{Z}$ the \textit{coarse value group}.

Every formally $p$-adic field embeds into a \textit{$p$-adic closure}, that is, an algebraic extension which is $p$-adically closed. The $p$-adic closure is not necessarily unique. The theory $\pCF$ is the model companion of the theory of formally $p$-adic fields, and hence every formula is equivalent, modulo $\pCF$, to an existential formula. In fact, $\pCF$ eliminates quantifiers after adding the predicate $P_n$ which picks out the $n$th powers \cite{Macintyre76}. Thus the elementary diagram of any computable model of $\pCF$ is decidable. We denote by $P_n^*$ the non-zero $n$th powers. The theory $\pCF$ also admits definable Skolem functions \cite{vandenDries84}. Finally, there is a cell decomposition theorem for definable sets in a $p$-adically closed field (see \cite{Denef86,ScowcroftvandenDries88,Mourgues09}).

\begin{definition}
The collections of \textit{cells} in $K$ is defined recursively by:
\begin{enumerate}%[label=(\arabic{*})]
	\item If $X$ is a single point in $K^n$, then $X$ is a (0)-cell.
	\item If $\Box_{1}$ and $\Box_{2}$ are either $<$, $\leq$, or no condition, $\gamma_1,\gamma_2 \in v(K) \cup \{-\infty, \infty\}$ $c \in K$, $k \in \omega$, and $\lambda \in K^\times$, then
	\[ \{ x \in K: \gamma_1 \Box_1 v(x-c) \Box_2 \gamma_2 \text{ and } P^*_{k}(\lambda(x-c))\} \]
	is a (1)-cell.
	\item If $f$ is a definable continuous function from a $(i_1,\ldots,i_n)$-cell $C$ to $K$, then the graph of $f$ is a $(i_1,\ldots,i_n,0)$-cell.
	\item If $B$ is a $(i_1,\ldots,i_n)$-cell, $f$, $g$, and $h$ are definable continuous functions from $B$ to $K$, $\lambda \in K^\times$, and $\Box_{1}$ and $\Box_{2}$ are either $<$, $\leq$, or no condition,	then
	\[ C = \{(\bar{x},y)\in B\times K:v(f(\bar{x}))\Box_{1}v(y-g(\bar{x}))\Box_{2}v(h(\bar{x}))\text{ and }P^*_{k}(\lambda(y-g(\bar{x}))\} \]
	is a $(i_1,\ldots,i_n,1)$-cell.
\end{enumerate}
\end{definition}

\begin{theorem}[Cell decomposition for $\pCF$]
Let $(K,v)$ be a $p$-adically closed valued field. Let $S \subseteq K^n$ be a definable set. Then $S$ can be partitioned into finitely many cells. Moreover, the parameters over which the cells are defined are all definable over the parameters of $S$.
\end{theorem}

\section{Extending valuations}

We begin this section by showing that we can effectively embed valued fields into their Henselizations and into algebraically closed valued fields. This result appeared in \cite{Smith} and we repeat the proof here as we will later build on these ideas.

\begin{proposition}[Theorem 3 of \cite{Smith}]\label{prop:embeds-acvf}
Let $(K,v)$ be a computable valued field. There is a computable embedding of $K$ into a computable presentation $\overline{K}$ of its algebraic closure and a computable extension of $v$ to $\overline{K}$.
\end{proposition}
\begin{proof}
If $v$ is the trivial valuation, then extend it to the trivial valuation on $\overline{K}$ under any computable embedding of $K$ into its algebraic closure. Otherwise, the theory $\ACVF \cup \text{Diag}_{\text{at}}(K)$ is complete, hence decidable. So it has a computable model $(L,w)$ by an effective Henkin construction (see, for example, \cite{Harizanov98}), and we get a computable embedding of $K$ into $L$ by mapping $x \in K$ to the interpretation of the constant representing $x$ in $L$. In $L$, we can enumerate the algebraic closure $\overline{K}$ of $K$ and hence construct a computable presentation.
\end{proof}

\noindent A consequence of this is that every computable non-trivially-valued field $K$ embeds into a model of $\ACVF$ whose underlying field is algebraic over $K$.

\begin{lemma}\label{lem:find-extensions}
Let $(K,v)$ be a computable finite extension of valued fields of $\mathbb{Q}$ with the $p$-adic valuation. Given $K(a)$ a finite field extension of $K$, we can compute a list of all of the extensions of $v$ to $K(a)$, with no duplication, as well as the ramification indices and residue degrees of these extensions. We can also compute the residue fields and the value groups of these extensions as subsets of $\overline{\mathbb{F}}_p = k_{\overline{\mathbb{Q}}}$ and $\mathbb{Q} = v(\overline{\mathbb{Q}})$ respectively. This computation is uniform in the generators for $K$ over $\mathbb{Q}$.
\end{lemma}
\begin{proof}
We argue by induction on the number of generators of $K$. Since we know the generators for $K$, $K$ has a splitting algorithm. Embedding $(K,v)$ into $(\overline{K},w) = (\overline{\mathbb{Q}},w)$ via the previous lemma, we can compute the image of $K$ in $\overline{K}$. We can compute the minimal polynomial of $a$ over $K$, and use it to find the embeddings of $K(a)$ into $\overline{K}$ over $K$. By restricting $w$ to $K(a)$ under each of these embeddings, we get a list of the possible extensions of $v$ to $K(a)$, possibly containing duplicates.

Given $u_1,\ldots,u_n$ valuations on $K(a)$ extending $v$, we claim that we can tell in a c.e.\ way that they are a complete list, without duplicates, of the extensions of $v$ to $K(a)$. To see that there are no duplicates in the list, we just have to find elements of $K(a)$ on which they differ. Since $K$ is a finite extension of $\mathbb{Q}$, $v(K) \cong \mathbb{Z}$, and so by Theorem \ref{thm:fund-ineq}, if $u_1,\ldots,u_n$ is a complete list of the extensions of $v$ to $K(a)$, then
\[ \sum_{i=1}^n e(u_i / v) f(u_i/v) = [K(a) : K]. \]
Note that we can compute $[K(a) : K]$ using the splitting algorithm for $K$. Inductively, we can compute the value group and residue field of $K$ as subsets of the value group $\mathbb{Q}$ and the residue field $\overline{\mathbb{F}}_p$ of $\overline{\mathbb{Q}}$ respectively. Since they are finitely generated substructures and we know the residue degree and ramification index of $(K,v)$ over $\mathbb{Q}$, we can compute finite sets of generators for the value group and residue field of $(K,v)$. So for each valuation $u$ from among $u_1,\ldots,u_n$, we can compute the value group and residue field of $u$ as c.e.\ subsets of $\mathbb{Q}$ and $\overline{\mathbb{F}}_p$. So we can compute increasing sequences with limits $e(u/v)$ and $f(u/v)$.

We always have, for any such list with no duplication,
\[ \sum_{i=1}^n e(u_i / v) f(u_i/v) \leq [K(a) : K]. \]
So $u_1,\ldots,u_n$ is a complete list if and only if the increasing approximations to $e(u_i / v)$ and $f(u_i/v)$ we computed above eventually give equality.

When we compute, in this way, a complete list of the extensions of $v$ to $K(a)$, we also get their ramifications indices and residue degrees. Using these values, we can compute the value groups and residue fields of these extensions as subsets of $\mathbb{Q}$ and $\overline{\mathbb{F}}_p$
\end{proof}

Let $(K,v)$ be a computable valued field with a splitting algorithm. Given an element $a$ algebraic over $K$, one can use Newton polygons to decide what possible valuations $a$ can take under an extension of $v$ to $K(a)$. Even if $a$ always has a unique valuation, $K(a)$ may admit multiple distinct extensions of $v$. The following lemma shows that in the general case (i.e., when $K$ is not finitely generated) there is no way to decide in a computable way, from the minimal polynomial of $a$ over $K$, how many extension of $v$ there are.

\begin{proposition}\label{prop:no-comp-ext}
There is a computable algebraic valued field $(K,v)$ with a splitting algorithm such that there is no way to (uniformly in $a$) compute the number of extensions of $v$ to an algebraic extension $K(a)$.
\end{proposition}
\begin{proof}
Assume that $0'_0 = \varnothing$, and that at each subsequent stage, exactly one element enters $0'$. Fix a presentation $\overline{K}$ of the algebraic closure of $K$ and a computable Rabin embedding of $K$ into $\overline{K}$.

Fix an odd prime $r$. Let $p_1,p_2,\ldots$ be a list of the infinitely many primes $p \neq r$. Begin at stage $0$ with $K_0 = \mathbb{Q}$ with $v_0$ the $r$-adic valuation.

Suppose that at stage $s$, $0'_s = \{a_1,\ldots,a_s\}$. We will have already defined
\[ K_s = \mathbb{Q}((r q_{i})^{\frac{1}{p_{a_i}}} : i = 1,\ldots,s )\]
with the unique extension $v_s$ of the $r$-adic valuation to $K_s$ (the fact that this extension of the valuation is unique follows from the fundamental inequality). Here, $q_1,\ldots,q_s$ are distinct primes $q \equiv 1 \mod r$. Let $a_{s+1} = b$ be the element which enters $0'$ at stage $s+1$. Search for a prime $q_{s+1} \equiv 1 \mod{r}$ which is not $r$ such that, as subsets of $\overline{K}$ with domain $\omega$,
\[ K_s((r q_{s+1})^{\frac{1}{p_{b}}}) \cap \{ 0,\ldots,s \} = K_s \cap \{0,\ldots,s\}.\]
Let $K_{s+1} = K_s((r q_{s+1})^{\frac{1}{p_{b}}})$. As $K_s$ is an extension of $\mathbb{Q}$ of degree $p_{a_1} \cdots p_{a_s}$ and $p_b$ is coprime to this, for any two distinct primes $q$ and $q'$,
\[ K_s((rq)^{\frac{1}{p_{b}}}) \cap K_s((r q')^{\frac{1}{p_{b}}}) = K_s \text{ or } K_s((rq)^{\frac{1}{p_{b}}}) = K_s((r q')^{\frac{1}{p_{b}}}).\]
Thus we can find a $q_{s+1}$ as desired. Extend $v_s$ to the unique valuation $v_{s+1}$ on $K_{s+1}$. Let $(K,v) = \bigcup_s (K_s,v_s)$. Note that $K$ has a splitting algorithm: to decide whether a give $s \in \overline{K}$ is in $K$, one can simply check whether $s \in K_s$. Also, $v$ is the unique extension of the $r$-adic valuation from $\mathbb{Q}$ to $K$.

We claim that if $a \in 0'$, then the valuation $v$ on $K$ has more than one extension to $K(r^{\frac{1}{p_a}})$, and if $a \notin 0'$, then $v$ has a unique extension to $K(r^{\frac{1}{p_a}})$.

First suppose that $a$ enters $0'$ at stage $s$. Then we have a tower of extensions
\[ \mathbb{Q} \subset \mathbb{Q}(1 + q_{s}^{\frac{1}{p_{a}}}) \subset K(r^{\frac{1}{p_{a}}}) .\]
Note that $1 + q_{s}^{\frac{1}{p_{a}}}$ has minimal polynomial
\[ (x-1)^{p_a} - q_s = \binom{p_a}{0}x^{p_a}-\binom{p_a}{1}x^{p_a-1}+\binom{p_a}{2}x^{p_a-2}+\cdots\pm\binom{p_a}{p_a-1}x - (q_s - 1).\]
Since $q_s \equiv 1 \mod{r}$, $r \mid q_s-1$. Also, since $p_a \neq r$, $r \nmid \binom{p_a}{p_a-1} = p_a$, $r \nmid \binom{p_a}{0} = 1$, and $r \nmid \binom{p_a}{1} = p$. Thus, by looking at the Newton polygon of this minimal polynomial, we see that there are multiple distinct extensions of the $r$-adic valuation on $\mathbb{Q}$ to $\mathbb{Q}(1 + q_{s}^{\frac{1}{p_a}})$. So there are multiple distinct extensions of the $r$-adic valuation on $\mathbb{Q}$ to $K(r^{\frac{1}{p_{a}}})$. Since $v$ was the unique extension of the $r$-adic valuation to $K$, there are multiple extensions of $v$ to $K(r^{\frac{1}{p_{a}}})$.

Now suppose that $a \notin 0'$. Then consider the tower of extensions
\[ \mathbb{Q} \subset \mathbb{Q}(r^{\frac{1}{p_{a}}}) \subset K_1(r^{\frac{1}{p_{a}}}) \subset K_2(r^{\frac{1}{p_{a}}}) \subset \cdots.\]
Since $a \notin 0'$, each of these extensions has ramification index equal to its degree as a field extension. By the fundamental inequality, there is a unique extension of the valuation for each field extension.
\end{proof}

We can also embed every valued field into a computable presentation of its Henselization.

\begin{proposition}[Proposition 6 of \cite{Smith}]
Let $(K,v)$ be a computable valued field. There is a computable embedding of $K$ into a computable valued field $(L,w)$ such that $(L,w)$ is the Henselization of $K$.
\end{proposition}
\begin{proof}
It is enough to show that if $(K,v) \to (\overline{K},v)$ is an embedding of $K$ into a computable presentation of its algebraic closure, then we can enumerate $K^h$ in $\overline{K}$. We can close under applications of Hensel's lemma, say in the version (2) of Definition \ref{def:Hensel} above, to enumerate the Henselization of $K$. Note that the solutions in (2) are unique.
\end{proof}

\noindent Smith also showed that Henselizations are recursively unique \cite{Smith}.

Given an embedding of a valued field into its algebraic closure, we might want to decide which elements of the algebraic closure are in the Henselization, rather than just enumerating the elements of the Henselization. We show that this can be done for the Henselization of $\mathbb{Q}$ inside any fixed presentation of $\overline{\mathbb{Q}}$.

\begin{proposition}
Let $(\mathbb{Q},v)$ be a computable valued field with the $p$-adic valuation. Fix a Rabin embedding of $\mathbb{Q}$ into $\overline{\mathbb{Q}}$. Then the Hensilization $\mathbb{Q}^h \subseteq \overline{\mathbb{Q}}$ of $\mathbb{Q}$ is computable inside $\overline{\mathbb{Q}}$.
\end{proposition}
\begin{proof}
Since $\mathbb{Q}$ is finitely ramified, it has value group $\mathbb{Z}$. By Theorem \ref{thm:alg-max}, the Henselization of $\mathbb{Q}$ is the smallest algebraically maximal valued field containing $\mathbb{Q}$; that is, the minimal extension of $\mathbb{Q}$ with no immediate extensions. The Hensilization of $\mathbb{Q}$ is unique inside the fixed presentation of $\overline{\mathbb{Q}}$. 

Given $a$, $a \in \mathbb{Q}^h$ if and only if $\mathbb{Q}(a)$, together with the induced valuation $v$ coming from the valuation on $\overline{\mathbb{Q}}$, is an immediate extension of $\mathbb{Q}$. If $a \in \mathbb{Q}^h$, then $\mathbb{Q}(a)$ is an immediate extension of $\mathbb{Q}$. On the other hand, if $\mathbb{Q}(a)$ is an immediate extension of $\mathbb{Q}$, since $\mathbb{Q}(a)^h$ is an immediate extension of $\mathbb{Q}(a)$ we know that $\mathbb{Q}(a)^h$ is an immediate extension of $\mathbb{Q}^h$. Since $\mathbb{Q}^h$ has no proper immediate extensions, $\mathbb{Q}(a)^h = \mathbb{Q}^h$. Thus $a \in \mathbb{Q}^h$. 

To check whether $\mathbb{Q}(a)$ is an immediate extension of $\mathbb{Q}$, we need to compute the ramification index and residue degree of the extension of $v$ to $\mathbb{Q}(a)$. We can do this uniformly in $a$ by Lemma \ref{lem:find-extensions}.  
\end{proof}

This lemma is not true for an arbitrary algebraic valued field. The following proposition shows that there is a computable algebraic valued field $(K,v)$, with a splitting algorithm, so that we cannot decide whether or not an element $a$ is in the Henselization of $K$. As a consequence, there is no computable way to decide, from a minimal polynomial of $a$ over $K$, whether or not $a$ is in the Henselization of $K$.

\begin{proposition}
There is a computable algebraic valued field $(K,v)$ with a splitting algorithm whose Henselization is not computable as a subset of $\overline{\mathbb{Q}}$.
\end{proposition}
\begin{proof}
Fix a prime $r$ and a computable list $p_1,p_2,\ldots$ of the primes not equal to $r$. In a similar way to Proposition \ref{prop:no-comp-ext}, construct a computable valued field
\[ (K,v) = \mathbb{Q}((r q_i)^{\frac{1}{p_i}} : i \in 0')\]
with a splitting algorithm. As before, for each $i$, $q_i \equiv 1 \mod{r}$. The primes $q_i$ do not necessarily form a computable sequence in $i$. The valuation $v$ is the unique extension of the $r$-adic valuation to $K$.

Then for each $i \in 0'$, $q_i^{\frac{1}{p_i}}$ is in the Henselization of $\mathbb{Q}$, and hence in the Henselization of $(K,v)$. This is because $1^{p_i} \equiv q \mod r$ but $p_i 1 ^{p_i - 1} \equiv p_i \not\equiv 0 \pmod r$. So $r^{\frac{1}{p_i}}$ is in the Henselization of $(K,v)$.

On the other hand, suppose that $i \notin 0'$. We will show that $r^{\frac{1}{p_i}}$ is not in the Henselization of $(K,v)$. Note that the value group of $K$ is $\mathbb{Z}\langle \frac{1}{p_j} : j \in 0' \rangle$. Then this is also the value group of the Henselization of $K$, and so $r^{\frac{1}{p_i}}$ is not in the Henselization. 
\end{proof}

We now come to the main result of this section. We showed above that we can embed a valued field $(K,v)$ into an algebraically closed valued field, constructing the algebraic closure $\overline{K}$ at the same as we construct the extension of the valuation. But what if we have a fixed embedding of $K$ into a presentation of its algebraic closure, and we want to extend the valuation $v$ to $\overline{K}$ via that particular embedding? Theorem 4 of \cite{Smith} shows that one cannot always do this.

If $\iota$ is an embedding of $K$ into $\overline{K}$, by an ($\iota$-)extension of the valuation $v$ to the field $\overline{K}$ we mean a valuation $w$ on $\overline{K}$ with $w \circ \iota = v$. The following theorem gives a necessary and sufficient condition for a valuation $v$ on an algebraic field $K$ to extend to every algebraic extension.

\begin{theorem}
Let $(K,v)$ be a computable algebraic valued field. Then the following are equivalent:
\begin{enumerate}
	\item for every computable embedding $\iota \colon K \to L$ of $K$ into a field $L$ algebraic over $K$, there is a computable extension of $v$ to a computable valuation $w$ on $L$,
	\item the \textit{Hensel irreducibility set}
\begin{align*}
H_{K} := \{ f = x^n + a_{n-1} x^{n-1} + a_{n-2} x^{n-2} + \cdots + a_0 \in \mc{O}_K[x] : \\ 
 f \text{ is irreducible over } K\text{, } v(a_{n-1}) = 0 \text{, and } v(a_{n-2}),\ldots,v(a_0) > 0 \}
\end{align*}
	of $(K,v)$ is computable.
\end{enumerate}
\end{theorem}

\noindent Note the relation between the set $H_K$ and (4) of Definition \ref{def:Hensel}. Indeed, Smith showed that given a Henselian computable field, and a fixed embedding in an algebraic closure, one can extend the valuation (see Proposition 5 of \cite{Smith}); our result can be seen as a significant generalization of this, as a Henselian field trivially has computable Hensel irreducibility set.

\begin{proof}
(2)$\Rightarrow$(1). Fix $\overline{\mathbb{Q}}$ a computable presentation of the algebraic closure of $\mathbb{Q}$, and $\iota$ an embedding of $K$ into $\overline{\mathbb{Q}}$. Using this embedding, we can view $K$ as a c.e.\ subset of $\overline{\mathbb{Q}}$. Begin by defining $w_0$ to be the $p$-adic valuation on $F_0 = \mathbb{Q}$.

We begin by showing that we can find a sequence
\[ F_0 = \mathbb{Q} \subseteq F_1 = F_0(a_0) \subseteq F_2 = F_1(a_1) \subseteq \cdots \]
of fields, such that each $F_s$ is a normal extension of $F_0$, and so that $\overline{\mathbb{Q}}$ is the union of these fields.
Given $F_s$ a finite normal extension of $\overline{\mathbb{Q}}$, and a splitting algorithm for $F_s$, $F_s$ is a computable subset of $\overline{\mathbb{Q}}$. Let $a$ be the first element of $\mathbb{Q}$ which is not in $F_s$. Search for an element $a_s$ such that $a \in F_i(a_s)$, and all of the conjugates of $a_s$ over $\mathbb{Q}$ are in $F_s(a_s)$. By Theorem \ref{Kronecker}, $F_s(a_s)$ has a splitting algorithm, so we can check this computably. Some such $a_s$ exists by the primitive element theorem. Then let $F_{s+1} = F_s(a_s)$. We have, uniformly in $s$, a splitting algorithm for $F_s$.

Suppose that we have defined $w_s$ on $F_s$, with the property that there is a common extension of $v$ and $w_s$ to $\overline{\mathbb{Q}}$. We will show how to extend $w_s$ to a valuation $w_{s+1}$ on $F_{s+1}$ such that $w_s$ and $v$ have a common extension to $\overline{\mathbb{Q}}$.

By Lemma \ref{lem:find-extensions} we can find all of the extensions of $w_s$ to $F_{s+1}$. If there is only one extension, let $w_{s+1}$ be this extension. Otherwise, let $u_1,\ldots,u_m$ be the distinct valuations on $F_{s+1}$ extending the $p$-adic valuation on $\mathbb{Q}$.

For each $i$, we will search for evidence that $u_i$ is not compatible with $v$. If $u_i$ is not compatible with $v$, then (by K\"onig's Lemma, since there are only finitely many valuations on a finitely generated algebraic extension of $\mathbb{Q}$) there is some finitely generated subfield $K'$ of $K$ such that $v \mid K'$ and $u_i$ are not compatible on $K' F_{s+1}$. For each $K'$, $K' F_{s+1}$ is a finite degree extension of $\mathbb{Q}$, and so by Lemma \ref{lem:find-extensions} we can find all of the valuations on $K' F_{s+1}$. If $u_i$ and $v  \mid K'$ do not have a common extension to $K' F_{s+1}$, then every valuation on $K' F_{s+1}$ will differ from either $u_i$ or $v  \mid K'$ when applied to some element. So if $u_i$ and $v$ are not compatible, we will discover this in a c.e.\ way.

On the other hand, using Theorem \ref{thm:distinguish} with $a_i = 1$ and $a_j = 0$ for $i \neq j$, there is $\beta_i \in F_{s+1}$ such that $u_i(\beta_i - 1) > 0$ and $u_j(\beta_i) > 0$ for $j \neq i$. Note that $u_i(\beta_i) = 0$. We can choose such a $\beta_i$ for each $i$. We claim that if $u_i$ and $v$ have a common extension, say $w$, to $K F_{s+1}$, then we can eventually find the minimal polynomial of $\beta_i$ over $K$. Let
\[ f_i = x^n + a_{n-1}x^{n-1} + a_{n-2}x^{n-2} + \cdots + a_0 \]
be the minimal polynomial of $\beta_i$ over $K$. Let $\beta_i = \beta_i^1,\ldots,\beta_i^n$ be the conjugates of $\beta_i$ over $K$. Since $F_{s+1}$ is a normal extension of $\mathbb{Q}$, and $\beta_i \in F_{s+1}$, each of these conjugates is in $F_{s+1}$. Each of $\beta_i^2,\ldots,\beta_i^n$ is a conjugate of $\beta_i$ over $\mathbb{Q}$. Among $(u_j)_{i \neq j}$ are the conjugates of the valuation $u_i$ over $\mathbb{Q}$. Since $u_j(\beta_i) > 0$ for each $i \neq j$, $u_i(\beta_i^2),\ldots,u_i(\beta_i^n) > 0$.
Then
\[ f_i = (x - \beta_i^1)(x-\beta_i^2) \cdots (x-\beta_i^n) \]
and so 
\[ v(a_{n-1}) = w(a_{n-1}) = w(-\beta_i^1-\cdots-\beta_i^n) = w(-\beta_i) = u_i(\beta_i) = 0. \]
For $k = 0,\ldots,n-2$, we can write $a_k$ as a sum of products of $\beta_i^1,\ldots,\beta_i^n$, where each term of the sum has at least two factors, and each of $\beta_i^1,\ldots,\beta_i^n$ shows up at most once in each product. Thus the $w$-value of each term is strictly positive, and so $v(a_k) = w(a_k) > 0$.
To find the minimal polynomial of $\beta_i$ over $K$, we search for an irreducible polynomial 
\[ f = x^n + a_{n-1}x^{n-1} + a_{n-2}x^{n-2} + \cdots + a_0 \]
with $f(\beta_i) = 0$, $v(a_{n-1}) = 0$, and $v(a_{n-2}),\ldots,v(a_0) > 0$. Note that we can check whether such a polynomial is irreducible. We can perform this search whether or not $u_i$ and $v$ have a common extension to $K F_{s+1}$. If $u_i$ and $v$ do have a common extension to $K F_{s+1}$, then we will eventually find the minimal polynomial of $\beta_i$. We can also find all of the conjugates $\beta_i = \beta_i^1,\ldots,\beta_i^n$ of $\beta_i$ over $K$.

Suppose that $u_i$ and $v$ have a common extension to $K F_i (a)$, and $u_j$ and $v$ have a common extension to $K F_i(a)$. Since any two extensions of $v$ to $K F_{s+1}$ are conjugate over $K$, $u_i$ and $u_j$ are conjugate over $K$. Thus $u_j(\beta_i^k) = 0$ for some $k$.

On the other hand, suppose that $u_j(\beta_i^k) = 0$ for some $k$. Note that $u_j$ is conjugate over $K$ to a valuation $u_j'$ with $u_j'(\beta_i) = 0$. By choice of $\beta_i$, $u_j' = u_i$. Thus $u_i$ and $u_j$ are conjugate over $K$. Then $u_i$ and $v$ have a common extension to $K F_{s+1}$ if and only if $u_j$ and $v$ do.

Eventually, we will find, for some $i$, the minimal polynomial
\[ f = x^n + a_{n-1}x^{n-1} + a_{n-2}x^{n-2} + \cdots + a_0 \]
of $\beta_i$ over $K$, and conjugates $\beta_i = \beta_i^1,\ldots,\beta_i^n$ of $\beta_i$ over $K$.
Some of the $u_j$, for $j \neq i$, will be found to be incompatible with $v$. The rest of the $u_j$ will have $u_j(\beta_i^k) = 0$ for some $k$. Since at least one of the $u_j$ has a common extension with $v$ to $K F_{s+1}$, it must be that $u_i$ and all of the $u_j$ with $u_j(\beta_i^k) = 0$ have such an extension. In particular, $u_i$ and $v$ have a common extension to $K F_{s+1}$. Take $w_{s+1} = u_i$.

\vspace{5pt}

For (1)$\Rightarrow$(2), let $\overline{\mathbb{Q}}$ be a computable presentation of the algebraic closure of $\mathbb{Q}$. Let $a_0,a_1,a_2,\ldots$ enumerate the elements of $K$. We will define, at stage $s + 1$, an embedding $\iota_{s+1} \colon \mathbb{Q}(a_0,\ldots,a_s) \to \overline{\mathbb{Q}}$ such that $\iota_0 \subseteq \iota_1 \subseteq \iota_2 \subseteq \cdots$. Then $\iota = \bigcup_s \iota_s$ will be an embedding of $K$ into $\overline{\mathbb{Q}}$. We will attempt to meet the following requirements:
\[ R_i : \varphi_i \text{ is not a valuation on } \overline{\mathbb{Q}} \text{ extending } v.\]
We know, by assumption, that we must fail to satisfy this requirement for some $i$, as there is a computable extension of $v$. We will use this failure to prove that $H_K$ is computable. The strategy is similar to that used in \cite{HTMelnikovMiller15}.

\vspace{5pt}

\noindent\textit{Construction.}

\vspace{5pt}

Begin with $\iota_0 \colon \mathbb{Q} \to \overline{\mathbb{Q}}$ the unique embedding. As stage $s+1$, we have already defined $\iota_{s} \colon \mathbb{Q}(a_0,\ldots,a_{s-1}) \to \overline{\mathbb{Q}}$. We must define $\iota_{s+1}$ on $\mathbb{Q}(a_0,\ldots,a_s)$.

Let $i \in \omega$ be least, if it exists, such that $R_i$ is not yet satisfied and there is a polynomial $f = x^n + b_{n-1} x^{n-1} + b_{n-2}x^{n-2} + \cdots + b_0 \in \mathbb{Q}(a_0,\ldots,a_{s-1})[x]$ with:
\begin{enumerate}
	\item $v(b_{n-1}) = 0$,
	\item $v(b_{n-2}),\ldots,v(a_0) > 0$, 
	\item $f$ is irreducible over $\mathbb{Q}(a_0,\ldots,a_{s-1})$,
	\item $f$ splits over $\mathbb{Q}(a_0,\ldots,a_{s})$, and
	\item for $c_1,\ldots,c_m$ the solutions of $\iota(f)$ in $\overline{\mathbb{Q}}$, $\varphi_{i,s}(c_j)$ is defined for each $j$.
\end{enumerate}
Since we have splitting algorithms for the finite extensions $\mathbb{Q}(a_0,\ldots,a_{s-1})$ and $\mathbb{Q}(a_0,\ldots,a_{s})$, we can check whether this is the case for a particular $f$. Since $\varphi_{i,s}$ converges for only finitely many inputs, there are only finitely many such $f$ to consider.

As the $b_i$ are symmetric functions in the roots of $f$, for any valuation of $\overline{\mathbb{Q}}$, all of the roots of $f$ have valuation $\geq 0$, and exactly one root of $f$ has valuation exactly zero. Suppose that $c_1,\ldots,c_r$ are the solutions of $\iota(f)$ with valuation $\leq 0$; note that there is at least one such solution, as otherwise we would have $a_{n-1} > 0$ since it is a sum of products of elements with valuation $> 0$. Then $b_{n-r}$ is the sum of the products of $r$ of the solutions of $\iota(f)$, and $c_1 \cdots c_r$ has (strictly) the least valuation among these; then $v(b_{n-r}) = v(c_1) + \cdots + v(c_r) \leq 0$. Hence $r = 1$, and $v(c_1) = v(b_{n-1}) = 0$. Without loss of generality, let $c_1,\ldots,c_m$ be the solutions of $\iota(f)$, with $\varphi_i(c_1) = 0$ and $\varphi_i(c_2),\ldots,\varphi_i(c_m) > 0$. (Note that if the valuations of the $c_i$ are different than this, then $\varphi_i$ is not a valuation of $\overline{\mathbb{Q}}$ $\iota$-extending $v$. Thus $R_i$ is satisfied.)

Now $f$ splits over $\mathbb{Q}(a_0,\ldots,a_s)$, say $f = g_1 \cdots g_\ell$ with $g_1,\ldots,g_\ell$ irreducible over $\mathbb{Q}(a_0,\ldots,a_s)$. Given the valuation $v$ on $\mathbb{Q}(a_0,\ldots,a_s) \subseteq K$, there is exactly one $j$ for which $g_j$ can have a solution with valuation $0$ with respect to a valuation extending $v$; we can find such a $j$ computably by looking at the values of the coefficients of the $g_j$. Without loss of generality, let $j = 1$. Note also that $g_1,\ldots,g_\ell$ are conjugate over $\mathbb{Q}(a_0,\ldots,a_{s-1})$. Thus, we can extend $\iota_s$ to $\iota_{s+1} \colon \mathbb{Q}(a_0,\ldots,a_s) \to \overline{\mathbb{Q}}$ such that $c_1$ is not a solution of $\iota(g_1)$. Then if $w = \varphi_i$ is a valuation on $\overline{\mathbb{Q}}$ $\iota$-extending $v$, $w(c_1) = 0$ and so $c_1$ must be a root of $\iota(g_1)$; but this is not the case, and so $w = \varphi_i$ is not a valuation on $\overline{\mathbb{Q}}$ $\iota$-extending $v$. Thus $R_i$ is satisfied.

\vspace{5pt}

\noindent\textit{End construction.}

\vspace{5pt}

We built an embedding $\iota = \bigcup_s \iota_s$ of $K$ into $\overline{\mathbb{Q}}$. By assumption, there is a computable valuation $w$ on $\overline{\mathbb{Q}}$ extending the valuation $v$ on $K$. Let $w$ be given by $\varphi_i$. Given a polynomial $f \in K[x]$, with $f = x^n + b_{n-1} x^{n-1} + b_{n-2}x^{n-2} + \cdots + b_0$ where $v(b_{n-1}) = 0$ and $v(b_{n-2}),\ldots,v(b_0) > 0$, let $c_1,\ldots,c_m$ be the solutions of $\iota(f)$ in $\overline{\mathbb{Q}}$. Let $t$ be a stage such that:
\begin{enumerate}
	\item no $R_j$, for $j < i$, acts after stage $t$,
	\item $\varphi_{i,t}(c_j)$ is defined for each $j$,
	\item $f \in \mathbb{Q}(a_0,\ldots,a_{t-1})[x]$.
\end{enumerate}
Note that (1) is independent of $f$, and depends only on the stage $i$. The following claim will finish the proof.

\begin{claim}
$f$ is irreducible over $K$ if and only if $f$ is irreducible over $\mathbb{Q}(a_0,\ldots,a_{t})$.
\end{claim}
\begin{proof}
The left to right direction is obvious. So suppose that $f$ is irreducible over $\mathbb{Q}(a_0,\ldots,a_{t})$. Then suppose that $f$ is not irreducible over $K$. Then $f$ splits over $\mathbb{Q}(a_0,\ldots,a_{s})$ for some least $s > t$. Then, by choice of $t$, in the construction we satisfy the requirement $R_i$ at stage $s+1$. But then $w$ does not extend $v$, a contradiction. So $f$ is irreducible over $K$.
\end{proof}

Given $f$, we can compute $t$ as required, and then to check whether $f$ is irreducible over $K$ it suffices to check whether it is irreducible over $\mathbb{Q}(s_0,\ldots,s_{t})$. Since this is a finite algebraic extension of $\mathbb{Q}$, we have a splitting algorithm for this field.
\end{proof}

\section{\texorpdfstring{$p$-adic closures}{p-adic closures}}

It was easy to see by an effective Henkin construction in Proposition \ref{prop:embeds-acvf} that every valued field embeds effectively into a computable algebraically closed valued field. The same argument does not work to show that every computable formally $p$-adic field embeds effectively into a computable $p$-adic closure, because the theory $\pCF$ is not the model completion of formally $p$-adic fields: if $K$ is a $p$-adic field, the elementary diagram of $K$ together with the theory $\pCF$ is not complete. Indeed, there is a computable formally $p$-adic field which does not computably embed into a $p$-adic closure. This uses ideas from the proof that a formally $p$-adic field whose value group is not a $\mathbb{Z}$-group embeds into two non-isomorphic $p$-adic closures (Theorem 3.2 of \cite{PrestelRoquette84}).

\begin{theorem}
There is a computable formally $p$-adic field which does not computably embed into a $p$-adic closure.
\end{theorem}
\begin{proof}
We will construct a formally $p$-adic field $E$ by diagonalizing against computable embeddings $f_i$ into $p$-adic closures $(K_i,v_i)$. Let $q_i$ be the $i$th prime. Begin at stage $0$ with $E_0 = \mathbb{Q}(t)$ a transcendental extension of $\mathbb{Q}$, together with the valuation $v$ with $v(t) > \mathbb{Z} = v(\mathbb{Q})$. At stage $s+1$, we will have built $E_0 \subseteq E_1 \subseteq \cdots \subseteq E_s$ a chain of embeddings of computable valued fields, with each extension algebraic. Let $i < s$ be the least $i$ against which we have not yet diagonalized such that at stage $s$ there is an element $a$ among the first $s$ elements of $K_i$ with $q_i \cdot v(a) = v(p^m f_i(t))$ for some $0 \leq m < q_i$ (i.e., $f_{i,s}(t)$ converges, and enough of the diagram of $K_i$ converges to decide that $q_i \cdot v(a) = v(p^m f_i(t))$). We will diagonalize against this $i$. Let $E_{s+1} = E_s(b)$, where $b$ is such that $b^{q_i} = p^{m+1} f_i(t)$. Extend the valuation to $E_{s+1}$ (again, by abuse of notation, calling it $v$).

Now $E_s$ will be an extension of degree $q_{i_1} \cdots q_{i_n}$ of $E_0$, where $i_1,\ldots,i_n$ are the requirements which we have already diagonalized against. $E_0(b)$ is an extension of $E_0$ of degree $q_i$, and so since $q_i$ is coprime to $q_{i_1},\ldots,q_{i_n}$, $E_{s+1}$ is an extension of $E_s$ of degree $q_i$.

The value group of $E_0$ is $\mathbb{Z}\langle r\rangle$, where $r = v(t) > \mathbb{Z}$. Then the value group of $E_s$ will be
\[ v(E_s) = \mathbb{Z}\langle r ,\frac{r + m_1+1}{q_{i_1}},\ldots,\frac{r + m_n+1}{q_{i_n}} \rangle.\]
The value group of $E_{s+1}$ will contain
\[ G = \mathbb{Z}\langle r,\frac{r + m_1 + 1}{q_{i_1}},\ldots,\frac{r + m_n + 1}{q_{i_n}},\frac{r + m + 1}{q_i} \rangle.\]
Since $q_i$ is coprime to $q_{i_1},\ldots,q_{i_n}$, $v(E_s)$ is a subgroup of $G$ index $q_i$. By the fundamental inequality, $G$ is the value group of $E_{s+1}$, and the residue degree is $1$. Note also that since $q_{i_1},\ldots,q_{i_n},q_i$ are coprime, $1 = v(p)$ is still the minimal element of the value group. So $E_s(b)$ is formally $p$-adic. The extension of $v$ from $E_s$ to $E_{s+1}$ is unique.

Let $E = \bigcup_i E_i$, with valuation $v$. Then $E$ is a formally $p$-adic field. Suppose towards a contradiction that $E$ computably embeds into a computable $p$-adic closure; let $i$ be an index such that the embedding is $f_i$ into the $p$-adic closure $(K_i,v_i)$. Since the value group of $K_i$ is a $\mathbb{Z}$-group, there is $\gamma \in \Gamma_{K_i}$ such that $q_i \gamma = v(f_i(t)) + m$ for some $0 \leq m < q_i$. Then, at some stage $s$ we have diagonalized against every $j < i$ which we will ever diagonalize against, there is $a$ among the first $s$ elements of $K_i$ with $q_i \cdot v(a) = v(p^m f_i(t))$ for some $0 \leq m < q_i$, $f_{i,s}(t)$ has converged, and enough of the diagram of $K_i$ has converged to decide that $q_i \cdot v(a) = v(p^m f_i(t))$. Then at stage $s + 1$, we will diagonalize against $K_i$ by putting into $E_{s+1}$ an element $b$ with $v(b^{q_i}) = v(t) + m + 1$. Then, in the value group of $K_i$, we have
\[ q_i \cdot v\left( \frac{f_i(b)}{a} \right) = q_i \cdot v(f_i(b)) - q_i \cdot v(a) = v(f_i(c)) - q_i \cdot v(a) = v(f_i(t)) + m + 1 - v(f_i(t)) - m = 1.\]
But then $q_i$ divides $1$ in the value group, and hence $K_i$ is not formally $p$-adic. So $K_i$ cannot be the $p$-adic closure of $E$.
\end{proof}

The problem with the field from the previous theorem which prevents us from embedding it into a $p$-adic closure is that we cannot decide, for a given element of the value group and $n \in \omega$, whether or not it is divisible by $n$. Theorem \ref{thm:div-extens} below will show that this is the only obstacle.

\begin{definition}
Let $G$ be a torsion-free abelian group. The \textit{dividing set} of $G$ is 
\[ \DIV(G) = \{ (x,n) \in G \times \omega : n \text{ divides } x \}. \]
\end{definition}

\noindent The set $\DIV(G)$ should be viewed as analogous to the splitting set for a field. The following lemma is the analogue of Theorem \ref{Kronecker} of Kronecker's.

\begin{lemma}\label{lem:ext-div}
Suppose that $G$ is a computable torsion-free abelian group with $\DIV(G)$ computable. Let $H = G \langle a \rangle$ be a computable group where $a$ is a new element with $n a = b \in G$. Then $\DIV(H)$ is computable uniformly in $\DIV(G)$, $n$, and $b$.
\end{lemma}

Note that since $G$ is torsion-free, $a$ is uniquely determined by $b$ and $n$.

\begin{proof}
We begin by finding $a \in H$ with $na = b$. Using $\DIV(G)$, we may suppose that $b$ is not divisible in $G$ by any prime factor of $n$; if it is, find such a divisor, and replace $b$ by that divisor. Thus $m a \notin G$ for any $0 \leq m < n$. Given $x \in H$, write $x = m a + g$ with $0 \leq m < n$ and $g \in G$. We want to decide whether $x$ is divisible by some number $r$. It suffices to decide whether $x$ is divisible by a prime $q$; if it is, then we can find such a divisor and repeat the process, noting that since $G$ is torsion-free, divisors are unique.

If $q$ and $n$ are coprime, then we claim that $q$ divides $x$ if and only if $q$ divides $mb + ng$. If $q$ divides $x$, then $q$ divides $mb + ng = nx$. For the other direction, suppose $q$ divides $mb + ng$, say $q h = mb + ng$. Since $q$ and $n$ are coprime, let $r$ and $s$ be such that $q r = 1 + n s$. Then
\[ q(r x - s h) = q r x - q s h = (1+ns)x - nsx = x. \]
So $q$ divides $x$. Since $mb + ng$ is in $G$, we can decide whether $q$ divides $mb + ng$, and hence whether $q$ divides $x$.

On the other hand, suppose that $q$ and $n$ are not coprime, so that $q \mid n$. If $q$ divides $g$ and $q \mid m$, then $q$ divides $x = ma + g$. For the other direction, suppose that $q$ divides $x = ma + g$. Let $y = m' a + g'$, with $g' \in G$, be such that $q y = x$. Then
\[ (qm' - m) a = g - q g'\]
Thus $n \mid qm' - m$, and so $q \mid m$. Since $q$ divides $m a$ and $q$ divides $ma + g$, $q$ divides $g$. So $q$ divides $x = ma + g$ in $H$ if and only if $q \mid m$ and $q$ divides $g$ in $G$.
Thus $\DIV(H)$ is computable.
\end{proof}

We are now ready to show that when we can compute the dividing set of the value group of a formally $p$-adic valued field, we can effectively embed the field into a $p$-adic closure.

\begin{theorem}\label{thm:div-extens}
Let $(K,v)$ be a computable formally $p$-adic valued field with value group $\Gamma$. Suppose that $\DIV(\Gamma)$ is computable. There is a computable embedding of $K$ into a computable $p$-adic closure $(L,w)$.
\end{theorem}
\begin{proof}
We will construct a sequence $(K_0,v_0) = (K,v)^h \subseteq (K_1,v_1) \subseteq (K_2,v_2) \subseteq \cdots$ of computable Henselian valued fields such that $(L,w) = \bigcup_i (K_i,v_i)$ is a computable $p$-adic closure of $(K,v)$. If $(a_i,q_i)$ is an enumeration of all of the pairs of elements $a$ from $L$ and primes $q$, with $a_i \in K_i$, we will ensure at stage $s+1$ that $q_s$ divides one of $w(a_s),w(a_s)+1,\ldots,w(a_s)+q_s - 1$. Note that we must construct the sequence $(a_i,q_i)$ concurrently with the $K_i$. For each $i$, $\DIV(v_i(K_i))$ will be computable.

At stage $s+1$, ask $\DIV(v_s(K_s))$ whether $q_s$ divides one of $v_s(a_s),v_s(a_s)+1,\ldots,v_s(a_s)+q_s - 1$. If it does, then just set $K_{s+1} = K_s$. Otherwise, let $b$ be an $q_s$th root of $a_s$ and let $E = K_s(b)$; since $K_s$ was Henselian, there is only extension $v'$ of $v_s$ to $E$. Note that $q_s$ divides $a_s$ in $v'(E)$. By Lemma \ref{lem:ext-div}, $\DIV(v(E))$ is computable uniformly. Let $K_{s+1}$ be a Henselization of $E$. Then the value group of $K_{s+1}$ is the same as that of $E$.

To see that $E$ (and hence $K_{s+1}$) is formally $p$-adic, we must show that the residue field is still $\mathbb{F}_p$ and that $1 = v(p)$ is still the least positive element of the value group. First, since $q_s$ is prime and $v(a_s)$ is not divisible by $q_s$, $[v'(E):v_s(K_s)] = q_s = [E:K_s]$. By the fundamental inequality, the residue degree of $v'$ over $v_s$ is one. Thus the residue field of $E$ is again $\mathbb{F}_p$.

Each element of $E$ can be written in the form
\[ d = c_{q_s - 1} b^{q_s - 1} + c_{q_s - 2} b^{q_s - 2} + \cdots + c_1 b + c_0 \]
with the $c_i \in K_s$. We want to show that $v'(d)$ is not strictly in between $0$ and $1 = v(p)$. Suppose to the contrary that $d$ has valuation strictly between $0$ and $1$. Note that as $v'(b),\ldots,v'(b^{q_s-1})$ are all distinct and not in $\Gamma_{K_s}$, that $v'(d) = \min_{0 \leq i \leq q_s - 1} v'(c_i b^i)$. Since $c_0 \in K_s$ does not have valuation strictly between $0$ and $1$, $v'(d) = v'(c_i b^i)$ for some $i \geq 1$. Then $0 < q_s v'(c_i b^i) < q_s$. Note that $q_s v'(c_i b^i) = q_s v(c_i) + i v(a_s)$ is in $\Gamma_{K_s}$. Let $\gamma = v(c_i) \in \Gamma_{K_s}$. Thus, for some $j$, $1 \leq j < q_s$, $q_s \gamma + i v(a_s) = j$. Since $1 \leq i < q_s$, $\gcd(q_s,i) = 1$. Let $m,n$ be such that $m i + n q_s = 1$. Then
\[ m q_s \gamma + v(a_s) = m q_s \gamma + m i v(a_s) + n q_s v(a_s) = m j + n q_s v(a_s).\]
Since $q_s \nmid m,j$, we can write $m j = q_s d - r$, where $1 \leq r < q_s$. Then
\[ v(a_s) + r = q_s (m \gamma + d + n v(a_s)).\]
This is a contradiction, as $q_s$ does not divide $v(a_s) + r$ in $\Gamma_{K_s}$. So no element of $E$ has valuation strictly between $0$ and $1$. Thus $E$ is formally $p$-adic.

Now $(L,w) = \bigcup_i (K_i,v_i)$ is a computable valued field into which $(K,v)$ embeds computably, and $(L,w)$ is algebraic over $(K,v)$. Moreover, $(L,w)$ is a model of $\pCF$: it is formally $p$-adic as the union of formally $p$-adic valued fields, it is Henselian as the union of Henselian fields, and we ensured that the value group was a model of Presburger arithmetic.
\end{proof}

\section{The Mal'cev Property}

We begin by recalling the metatheorem from \cite{HTMelnikovMontalban15}. The metatheorem is stated using the general notion of a pregeometry, but for the purposes of this paper, the pregeometry will always be algebraic independence in fields, and the reader need not know the general definition of a pregeometry.

\begin{definition}
A class $\mathcal{K}$ has the \emph{Mal{\textquotesingle}cev property} if each member $\mathcal{M}$ of $\mathcal{K}$ of infinite dimension  has a computable presentation $\mc{G}$ with a computable basis and a computable presentation $\mc{B}$ with no computable basis  such that $\mc{B} \cong_{\Delta^0_2} \mc{M} \cong_{\Delta^0_2} \mc{G}$.
\end{definition}

In \cite{HTMelnikovMontalban15}, two conditions were isolated which imply the Mal{\textquotesingle}cev property. We require some definitions before we state these conditions and the metatheorem.

\begin{definition} The \emph{independence diagram} $\mc{I}_{\mc{M}}({\bar{c}})$ of $\bar{c}$ in $\mc{M}$ is the collection of all existential formulas true of tuples independent over $\bar{c}$.\end{definition}

\begin{definition} We say that \emph{dependent elements are dense in $\mc{M}$} if, whenever $\mc{M} \models \exists \bar{y} \psi(\bar{c}, \bar{y}, a)$ for a quantifier-free formula $\psi$, non-empty tuple $\bar{c}$, and $a \in \mc{M}$, there is a $b\in \cl(\bar{c})$ such that $\mc{M} \models \exists \bar{y} \psi(\bar{c}, \bar{y}, b)$. We may also assume that $\bar{c}$ contains at least $m$ independent elements, for some fixed $m$.\end{definition}

\begin{definition}\label{defn:loc} We say that independent tuples in $\mc{M}$ are \emph{locally indistinguishable} if for every tuple $\bar{c}$ in $\mc{M}$ and $\bar{u}$, $\bar{v}$ independent tuples over $\bar{c}$,  
for each existential formula $\phi$ such that $\mc{M} \models \phi(\bar{c}, \bar{u})$, there exists a tuple $\bar{w}$ that is independent over $\bar{c}$, has $\mc{M} \models \phi (\bar{c}, \bar{w})$, and (with $\bar{w} = (w_1,\ldots,w_n)$ and $\bar{v} = (v_1,\ldots,v_n)$) we have $w_i \in \cl (\bar{c},v_1,\ldots,v_i)$ for $i = 1,\ldots,n$.
\end{definition}

The two conditions are as follows:

\medskip

\noindent  \textbf{Condition G:} Independent tuples are locally indistinguishable in $\mc{M}$ and for each $\mc{M}$-tuple $\bar{c}$, $\mc{I}_{\mc{M}}({\bar{c}})$ is computably enumerable uniformly in~$\bar{c}$.

\medskip

\noindent \textbf{Condition B:} Dependent elements   are dense in $\mc{M}$.

\medskip

\begin{theorem}[Theorem 1.2 of \cite{HTMelnikovMontalban15}]\label{metatheorem}
Let $\mc{K}$ be a class of computable structures that admits a r.i.c.e.~pregeometry $\cl$.\footnote{Recall that here this will just be algebraic independence.}
If  each $\mc{M}$ in $\mc{K}$ of infinite dimension satisfies Conditions $G$ and $B$, then $\mc{K}$ has the  Mal{\textquotesingle}cev property.
\end{theorem}

\subsection{\texorpdfstring{The Mal{\textquotesingle}cev property for $\ACVF$}{The Mal'cev property for ACVF}}

We will now use the metatheorem to show that algebraically closed valued fields have the Mal{\textquotesingle}cev property. Note that in $\ACVF$, algebraic dependence is the same as model-theoretic $\acl$.

\begin{theorem}
Algebraically closed valued fields %with Archimedean value group 
have the Mal{\textquotesingle}cev property.
\end{theorem}
\begin{proof}
Let $(K,v)$ be an algebraically closed valued field of infinite transcendence degree. We begin by checking that independent types are locally indistinguishable. Let $S \subseteq K^n$ be a definable set with parameters $\bar{c}$ which contains a tuple $\bar{a} = (a_1,\ldots,a_n) \in K^n$ independent over $\bar{c}$. We may assume that some element of the tuple $\bar{c}$ is non-trivially valued. Using quantifier elimination in $\ACVF$ and writing $S$ in disjunctive normal form, we may, without loss of generality, take $S$ to be the disjunct which contains $\bar{a}$. Since $S$ contains $\bar{a}$ which is independent over $\bar{c}$, $S$ is defined by a conjunction of formulas of the form $v(f(\bar{x},\bar{c}))\leq v(g(\bar{x},\bar{c}))$ (or such a formula with $\leq$ replaced by $<$, or $\neq$). The subfield $\mathbb{Q}(\bar{c})^{alg}$ is a model of $\ACVF$, and by model completeness, an elementary submodel of $K$. Hence it contains an element $\bar{u}=(u_1,\ldots,u_n)$ which is in $S$. Note that $S$ is open in the valuation topology, and so it contains an open ball
\[ B(\bar{u},\epsilon) = \{ \bar{x} : v(u_i - x_i) \geq \epsilon \} \]
around $\bar{u}$, with $\epsilon \in \Gamma(\mathbb{Q}(\bar{c})^{alg})$. % THIS WORKS?
There is also some $\bar{v} \neq \bar{u}$ with $\bar{v} \in B(\bar{u},\epsilon) \cap \mathbb{Q}(c)^{alg}$. Write $\bar{v} = (v_1,\ldots,v_n)$. Let $\bar{b} = (b_1,\ldots,b_n)$ be an arbitrary tuple from $K$ independent over $\bar{c}$. Possibly replacing each $b_i$ with $b_i^{-1}$, we may assume that $v(b_i) \geq 0$. Let $\bar{b}' = (b_iv_i - (b_i - 1)u_i)_{i=1}^n$. Note that $b_i$ and $b_i'$ are interalgebraic over $\bar{c}$. Then
\[ v(u_i - b_iv_i + (b_i - 1)u_i) = v(b_iu_i - b_iv_i) = v(b_i) + v(u_i-v_i) \geq \epsilon.\]
So $\bar{b}' \in B(\bar{u},\epsilon) \subseteq S$. We have shown that independent types are locally indistinguishable.

% Let $\bar{b} = (b_1,\ldots,b_n)$ be an arbitrary tuple from $K$ independent over $\bar{c}$. Since $K$ has Archimedean value group, for some $n \in \mathbb{Z}$, $v(b_i^{n}) > \epsilon$; replacing $b_i$ by $b_i^n$, we may assume that for each $i$, $v(b_i) < \epsilon$. Then the tuple $(a_1'+b_1,\ldots,a_n'+b_n)$ is in $B(\bar{a}',\epsilon)$ and hence in $S$. So independent types are locally indistinguishable.

A similar argument works to show that independent types are non-principal. Let $S \subseteq K^n$ be a definable set with parameters $\bar{c}$, again assuming that some element of the tuple $\bar{c}$ is non-trivially valued. Then $\mathbb{Q}(\bar{c})^{alg}$ is a model of $\ACVF$ and by model completeness there is a tuple $\bar{a} \in \mathbb{Q}(\bar{c})^{alg}$ which is contained in $S$. The tuple $\bar{a}$ is algebraic over $\bar{c}$.

We showed above that a definable set $S$ with parameters $\bar{c}$ contains a tuple independent over $\bar{c}$ if and only if it contains, as a disjunct, a non-empty definable set defined by a conjunction of formulas of the form $v(f(\bar{x},\bar{c}))\leq v(g(x,\bar{c}))$ (or with $\leq$ replaced by $<$ or $=$). Together with the decidability of the theory $\ACVF$, this fact allows us to enumerate the independence diagram of $K$.

By Theorem \ref{metatheorem}, $\ACVF$ has the Mal{\textquotesingle}cev property.
\end{proof}

\subsection{\texorpdfstring{The Mal'cev property for $\pCF$}{The Mal'cev property for pCF}}

Now we will apply the metatheorem to $p$-adically closed fields. Once again, the pregeometry will be algebraic independence which is the same as model-theoretic $\acl$. Our proof will use the cell decomposition for $p$-adically closed fields. We begin with a lemma which we will use to check that independent tuples are locally indistinguishable.

\begin{lemma}\label{lem:inter-alg}
Given a cell
\[ C=\{(\bar{x},y)\in B\times K:v(f(x))\Box_{1}v(y-g(\bar{x}))\Box_{2}v(h(x))\text{ and }P^*_{k}(\lambda(x-g(y))\} \]
and $\bar{a}\in B$, $b$ algebraically independent from $\bar{a}$, $\lambda$, and the coefficients of $f$ and $g$, with $(\bar{a},b)\in C$, and $c$ is algebraically independent from $\bar{a}$, there is $c'$ interalgebraic with $c$ over $\bar{a}$ with $(\bar{a},c')\in C$.
\end{lemma}
\begin{proof}
Since $b$ is algebraically independent from $\bar{a}$, we know that $k\neq0$. Assume that $\Box_1$ and $\Box_2$ are $\leq$, so that
\[ C=\{(\bar{x},y)\in B\times K:v(f(x))\leq v(x-g(y))\leq v(h(x))\text{ and }P^*_{k}(\lambda(x-g(y))\}. \]
The other cases are similar. It suffices to find $c''$ interalgebraic with $c$ over $\bar{a}$ such that $v(\lambda f(\bar{a}))\leq kv(c'')\leq v(\lambda h(\bar{a}))$, as then $c' = (c'')^k / \lambda + g(\bar{a})$ has $(\bar{a},c') \in C$.
We may replace $\lambda f$ by $\hat{f}$ and similarly with $h$ and $\hat{h}$ to get $v(\hat{f}(\bar{a}))\leq kv(c'')\leq v(\hat{h}(\bar{a}))$. Now $K\models(\exists y)v(\hat{f}(\bar{a}))\leq kv(y)\leq v(\hat{h}(\bar{a}))$, and so since we have definable Skolem functions, there is $a'$ algebraic over $\bar{a}$ satisfying this. Moreover, we can choose $a' \neq 0$.

If $v(c)=0$, then we have $v(ca')=v(a')$ and so we can take $c'' = ca'$. Otherwise, by replacing $c$ by $c^{-1}$ if necessary, we may assume that $v(c)>0$. Then $v(1+c)=0$, and so $v(a'+ca')=v(a')$. Then we can take $c'' = a'+ca'$.
\end{proof}

\begin{theorem}
$p$-adically closed fields have the Mal{\textquotesingle}cev property.
\end{theorem}
\begin{proof}
Let $(K,v)$ be a model of $\pCF$ of infinite transcendence degree. We begin by checking that independent tuples are locally indistinguishable. Let $S$ be a set definable over parameters $\bar{c}$, containing a tuple $\bar{a} = (a_1,\ldots,a_n)$ independent over $\bar{c}$. Let $\bar{b} = (b_1,\ldots,b_n)$ be another tuple independent over $\bar{c}$. The set $S$ has a cell decomposition with parameters definable over $\bar{c}$. Some cell must contain $\bar{a}$, and this cell must be of type $(1,\ldots,1)$ since $\bar{a}$ is independent over $\bar{c}$. By repeated applications of Lemma \ref{lem:inter-alg}, we get $\bar{b}'$ in $S$ as required.

Suppose that $S \subseteq K^n$ is a definable set over parameters $\bar{c} \in K^m$. Models of $\pCF$ have definable Skolem functions, so there is a definable function $f \colon K^m \to K^n$ (without parameters) with $f(\bar{c}) \in S$. Then $f(\bar{c})$ is definable over $\bar{c}$, and hence algebraic over $\bar{c}$. So independent types are non-principal.

Finally, we have to enumerate the independence diagram of $K$. We showed above that there
is an independent tuple in a cell if and only if it is of type $(1,\ldots,1)$. Using the decidability of the elementary diagram of $K$, we can enumerate the definable sets which contain such a cell.

By Theorem \ref{metatheorem}, $\pCF$ has the Mal{\textquotesingle}cev property.
\end{proof}

\bibliography{References}
\bibliographystyle{alpha}

\end{document}